\documentclass{llncs}

\usepackage[T1]{fontenc}

\usepackage{amsmath,graphicx}
\date{}

\newtheorem{observation}{Observation}
\newtheorem{claimlabelled}{Claim}

\title{
Unavoidable Patterns in Complete Simple Topological Graphs\thanks{Supported by NSF CAREER award DMS-1800746 and NSF award DMS-1952786.}
}

\author{
Andrew Suk\and Ji Zeng
}

\institute{
Department of Mathematics, University of California at San Diego\\ La Jolla, CA, 92093 USA\\ \email{\{asuk,jzeng\}@ucsd.edu}
}

\begin{document}

\maketitle

\begin{abstract}
We show that every complete $n$-vertex simple topological graph contains a topological subgraph on at least $(\log n)^{1/4 - o(1)}$ vertices that is weakly isomorphic to the complete convex geometric graph or the complete twisted graph. This is the first improvement on the bound $\Omega(\log^{1/8}n)$ obtained in 2003 by  Pach, Solymosi, and T\'oth.  We also show that every complete $n$-vertex simple topological graph contains a plane path of length at least $(\log n)^{1 -o(1)}$.

\keywords{Topological graph \and Unavoidable Patterns \and Plane path.}
\end{abstract}

\section{Introduction}

A \emph{topological graph} is a graph drawn in the plane or, equivalently, on the sphere, such that its vertices are represented by points and its edges are represented by non-self-intersecting arcs connecting the corresponding points. The arcs are not allowed to pass through vertices different from their endpoints, and if two edges share an interior point, then they must properly cross at that point in common. A topological graph is \emph{simple} if every pair of its edges intersect at most once, either at a common endpoint or at a proper crossing point. If the edges are drawn as straight-line segments, then the graph is said to be \emph{geometric}. If the vertices of a geometric graph are in convex position, then it is called \emph{convex}.  

Simple topological graphs have been extensively studied \cite{jan,handbook,pt05,rst,sw}, and are sometimes referred to as \emph{good drawings} \cite{oswin,oswin2022}, or simply as \emph{topological graphs} \cite{pst}. In this paper, we are interested in finding large unavoidable patterns in complete simple topological graphs. Two simple topological graphs $G$ and $H$ are \emph{isomorphic} if there is a homeomorphism of the sphere that transforms $G$ to $H$.  We say that $G$ and $H$ are \emph{weakly isomorphic} if there is an incidence preserving bijection between $G$ and $H$ such that two edges of $G$ cross if and only if the corresponding edges in $H$ cross as well.  Clearly, any two complete convex geometric graphs on $m$ vertices are weakly isomorphic.  Hence, let $C_m$ denote any complete convex geometric graph with $m$ vertices.

By the famous Erd\H os-Szekeres convex polygon theorem \cite{es} (see also \cite{sukes}), every complete $n$-vertex geometric graph contains a geometric subgraph on $m = \Omega(\log n)$ vertices that is weakly isomorphic to $C_m$. (Note that no three vertices in a complete geometric graph are collinear.) Interestingly, the same is not true for simple topological graphs.  The \emph{complete twisted graph} $T_m$ is a complete simple topological graph on $m$ vertices with the property that there is an ordering on the vertex set $V(T_m) = \{v_1,v_2,\dots, v_m\}$ such that edges $v_iv_j$ and $v_kv_{\ell}$ cross if and only if $i < k < \ell < j$ or $k < i< j < \ell$.  See Figure \ref{fig:ct}. It was first observed by Harborth and Mengerson \cite{HM} that $T_m$ does not contain a topological subgraph that is weakly isomorphic to $C_5$.  However, in 2003, Pach, Solymosi, and T\'oth \cite{pst} showed that it is impossible to avoid both $C_m$ and $T_m$ in a sufficiently large complete simple topological graph.

\begin{figure}
    \centering
    \includegraphics[scale=.55]{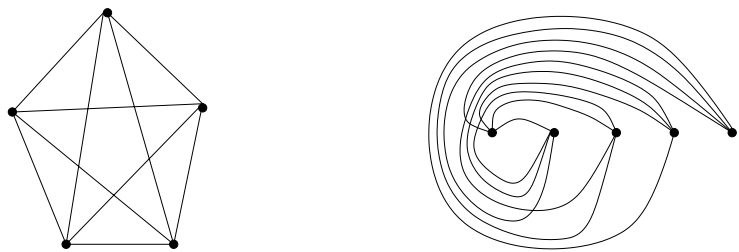}
    \caption{$C_5$ and $T_5$.}
    \label{fig:ct}
\end{figure}

\begin{theorem}[Pach-Solymosi-T\'oth]\label{old}
Every complete $n$-vertex simple topological graph contains a topological subgraph on  $m \geq \Omega(\log^{1/8}n)$ vertices that is weakly isomorphic to $C_m$ or $T_m$.  
\end{theorem}

The main result of this paper is the following improvement.
\begin{theorem}\label{main}
Every complete $n$-vertex simple topological graph has a topological subgraph on $m\geq (\log n)^{1/4 - o(1)}$ vertices that is weakly isomorphic to $C_m$ or $T_m$.
\end{theorem}

\noindent In the other direction, let us consider the following construction. Let $V = \{1,2,\dots, n\}$ be $n$ vertices placed on the $x$-axis, and for each pair $\{i,j\} \in V$, draw a half-circle connecting $i$ and $j$, with this half-circle either in the upper or lower half of the plane uniformly randomly.  By applying the standard probabilistic method \cite{alon}, one can show that there is a complete $n$-vertex simple topological graph that does not contain a topological subgraph on $m = \lceil 8\log n\rceil$ vertices that is weakly isomorphic to $C_m$ or $T_m$.  Another construction, observed by Scheucher \cite{Sch}, is to take $n$ points in the plane with no $2\lceil \log n\rceil$ members in convex position, and then draw straight-line segments between all pairs of points.

It is not hard to see that both $C_m$ and $T_m$ contain a plane (i.e. crossing-free) subgraph isomorphic to any given tree $T$ with at most $m$ vertices (see, e.g., \cite{gmpp}).  Thus, as a corollary of Theorem \ref{main}, we obtain the following.

\begin{corollary}\label{cor}
Every complete $n$-vertex simple topological graph contains a plane subgraph isomorphic to any given tree $T$ with at most $(\log n)^{1/4 - o(1)}$ vertices.
\end{corollary}

\noindent  In the case when $T$ is a path, we improve this bound with the following result, which is also recently obtained in \cite{oswin2022} independently.

\begin{theorem}\label{path}
Every complete $n$-vertex simple topological graph contains a plane path of length at least $(\log n)^{1 -o(1)}$.
\end{theorem}

In order to avoid confusion between topological and combinatorial edges, we write $uv$ when referring to a topological edge in the plane, and write $\{u,v\}$ when referring to an edge (pair) in a graph. Likewise, we write $\{u_1,\dots, u_k\}$ when referring to an edge ($k$-tuple) in a $k$-uniform hypergraph.  We systemically omit floors and ceilings whenever they are
not crucial for the sake of clarity in our presentation. All logarithms are in base 2.

\section{Monotone Paths and Online Ramsey Numbers}

Before we prove Theorem \ref{main}, let us recall the following lemmas. 
Let $H$ be a $k$-uniform hypergraph with vertex set $[n]  =\{1,2,\dots, n\}$. We say that $H$ contains a \emph{monotone $k$-path} of length $m$ if there are $m$ vertices $v_1 < v_2 < \cdots <  v_{m}$ such that $\{v_i,v_{i + 1},\dots, v_{i + k - 1}\}\in E(H)$ for $1 \leq i \leq m - k + 1$.  We say that the edge set $E(H)$ is \emph{transitive} if for any $v_1<v_2<\cdots<v_{k+1}$ in $[n]$, the condition $\{v_1,v_2,\dots, v_{k}\},\{v_2,v_3,\dots, v_{k + 1}\}\in E(H)$ implies all $k$-element subsets of $\{v_1,\dots, v_{k  + 1}\}$ are in $E(H)$.
We will need the following lemma due to Fox, Pach, Sudakov, and Suk.
\begin{lemma}[\cite{fpss}]\label{complete}
Let $n > k$, and let $H$ be a $k$-uniform hypergraph with vertex set $[n]$, which
contains a monotone path of length $n$, that is, $\{i, i+1,\dots , i+k-1\}\in  E(H)$ for all $1 \leq i \leq n -k+1$.  If $E(H)$ is transitive, then $H$ is the complete k-uniform hypergraph on $[n]$.
\end{lemma}

Next, we need a lemma from Online Ramsey Theory.  The \emph{vertex online Ramsey game} is a game played by two players, \emph{builder} and \emph{painter}. Let $t\geq 1$ and suppose vertices $v_1,v_2,\dots, v_{t-1}$ are present. At the beginning of stage $t$, a new vertex $v_t$ is added. Then for each $v_i \in \{v_1,\dots ,v_{t-1}$\}, builder decides (in any order) whether to create the edge $\{v_i,v_t\}$. If builder creates the edge, then painter has to immediately color it red or blue. When builder decides not to create any more edges, stage $t$ ends and stage $t+1$ begins by adding a new vertex. Moreover, builder must create at least one edge at every stage except for the first one. The \emph{vertex online Ramsey number} $r(m)$ is the minimum number of edges builder has to create to guarantee a monochromatic monotone path of length $m$ in a vertex online Ramsey game. Clearly, we have $r(m) \leq O(m^4)$, which is obtained by having builder create all possible edges at each stage and applying Dilworth's theorem \cite{dilworth} on the $m^2$ vertices. Fox, Pach, Sudakov, and Suk proved the following.
\begin{lemma}[\cite{fpss}]\label{online}
We have $r(m) = (1 + o(1))m^2 \log_2 m$.
\end{lemma}

\section{Convex Geometric Graph versus Twisted Graph}

In this section, we prove the following theorem, from which Theorem \ref{main} quickly follows.

\begin{theorem}\label{main2}

Let $m_1,m_2,n$ be positive integers such that

\[9(m_1 m_2)^2\log(m_1)\log(m_2) < \log n.\]

\noindent Then every complete $n$-vertex simple topological graph contains a topological subgraph that is weakly isomorphic to $C_{m_1}$ or $T_{m_2}$.

\end{theorem}
\begin{proof} Let $G = (V, E)$ be a complete $n$-vertex simple topological graph. Notice that the edges of $G$ divide the plane into several cells (regions), one of which is unbounded. We can assume that there is a vertex $v_0 \in  V$ such that $v_0$ lies on the boundary of the unbounded cell. Indeed, otherwise we can project $G$ onto a sphere, then choose an arbitrary vertex $v_0$ and then project $G$ back to the plane such that $v_0$ lies on the
boundary of the unbounded cell, moreover, the new drawing is isomorphic to the original one as topological graphs.

Consider the topological edges emanating out from $v_0$, and label their endpoints $v_1,\dots , v_{n-1}$ in
clockwise order.  For convenience, we write $v_i \prec v_j$ if $i < j$.  Given subsets $U,W \subset \{v_1,\dots, v_{n-1}\}$, we write $U\prec W$ if $u \prec w$ for all $u \in U$ and $w \in W$.  Following the notation used in \cite{pst}, we color the triples of $\{v_1,\dots, v_{n-1}\}$ as follows. For $v_i \prec v_j \prec v_k$, let $\chi(v_i,v_j,v_k) = xyz$, where $x,y,z \in\{0,1\}$ such that
\begin{enumerate}
    \item setting $x = 1$ if edges $v_jv_k$ and $v_0v_i$ cross, and let $x = 0$ otherwise;

    \item setting $y = 1$ if edges $v_iv_k$ and $v_0v_j$ cross, and let $y = 0$ otherwise;

    \item setting $z = 1$ if edges $v_iv_j$ and $v_0v_k$ cross, and let $z = 0$ otherwise.
\end{enumerate}

\noindent Pach, Solymosi, and T\'oth observed the following. \begin{observation}[\cite{pst}]\label{observation_pst}
The only colors that appear with respect to $\chi$ are $000$, $001$, $010$, and $100$.
\end{observation}

\begin{figure}
    \centering
    \includegraphics[scale=.38]{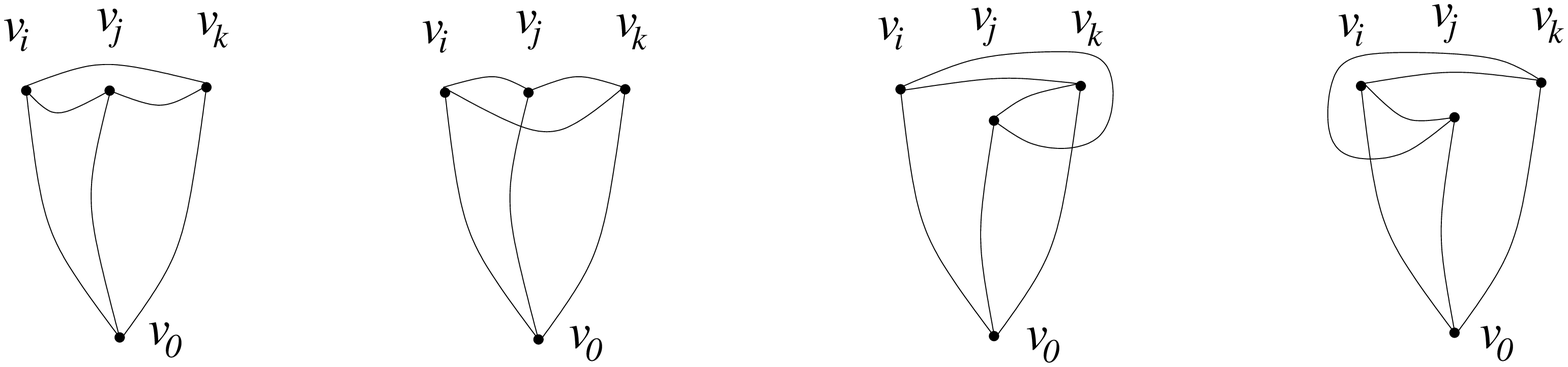}
    \caption{Configurations for 000, 010, 001, 100 respectively.}
    \label{fig:colors}
\end{figure}

See Figure \ref{fig:colors} for an illustration. We now make another observation.

\begin{lemma}\label{100}  Colors $001$ and $100$ are transitive.  That is, for $v_{i}\prec v_{j}\prec  v_{k} \prec v_{\ell}$,\begin{enumerate}
    \item if $\chi(v_{i},v_{j},v_{k}) = \chi(v_{j},v_{k},v_{\ell}) = 001$, then $\chi(v_{i},v_{j},v_{\ell}) = \chi(v_{i},v_{k},v_{\ell}) = 001$;

    \item  if $\chi(v_{i},v_{j},v_{k}) = \chi(v_{j},v_{k},v_{\ell}) = 100$, then $\chi(v_{i},v_{j},v_{\ell}) = \chi(v_{i},v_{k},v_{\ell}) = 100$.
\end{enumerate}
\end{lemma}

\begin{proof}  Suppose $\chi(v_{i},v_{j},v_{k}) = \chi(v_{j},v_{k},v_{\ell}) = 001$.  Since edges $v_0v_{\ell}$ and $v_jv_k$ cross, vertex $v_{\ell}$ must lie in the closed region bounded by edges $v_jv_k$, $v_iv_j$, and $v_0v_k$.  See Figure \ref{t001}.  Hence, edge $v_0v_{\ell}$ crosses both $v_iv_j$ and $v_iv_k$.  Therefore, we have $\chi(v_{i},v_{j},v_{\ell}) = \chi(v_{i},v_{k},v_{\ell}) = 001$ as wanted.  If $\chi(v_{i},v_{j},v_{k}) = \chi(v_{j},v_{k},v_{\ell}) = 100$, a similar argument shows that we must have $\chi(v_{i},v_{j},v_{\ell}) = \chi(v_{i},v_{k},v_{\ell}) = 100$.\qed\end{proof}

\begin{figure}[ht]
    \centering
    \includegraphics[scale=.5]{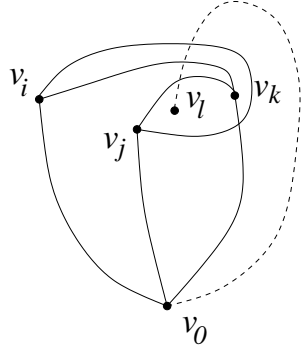}
    \caption{The closed region bounded by edges $v_jv_k$, $v_iv_j$, and $v_0v_k$ in Lemma \ref{100}.}
    \label{t001}
\end{figure}

Based on the coloring $\chi$, we define a coloring $\phi$ of the pairs of $\{v_1,v_2,\dots, v_{n-1}\}$ as follows. For $v_i \prec v_j$, let  $\phi(v_i,v_j) = (a,b)$ where $a$ is the length of the longest monotone 3-path ending at $\{v_i,v_j\}$ in color $100$, and $b$ is the length of the longest monotone 3-path ending at $\{v_i,v_j\}$ in color $001$. We can assume that $a,b < m_2$. Otherwise, by Lemma~\ref{100} and Lemma~\ref{complete}, we would have a subset $U\subset V$ of size $m_2$ whose triples are all of the same color, $100$ or $001$. And it is not hard to argue by induction that such a $U$ corresponds to a topological subgraph that is weakly isomorphic to $T_{m_2}$ as wanted.

\medskip
 
Before we continue, let us give a rough outline of the rest of the proof.  In what follows, we will construct disjoint vertex subsets $V^{a,b} \subset\{v_1,\dots, v_{n-1}\}$, where $1 < a,b < m_2$, such that $\phi$ colors every pair in $V^{a,b}$ with color $(a,b)$. For each $V^{a,b}$, we will play the vertex online Ramsey game by letting the builder create an edge set $E^{a,b}$ and designing a painter's strategy, which gives rise to a coloring $\psi$ on $E^{a,b}$. We then apply Lemma~\ref{online} to show that if $n$ is sufficiently large, some vertex set $V^{a,b}$ will contain a monochromatic monotone 2-path of length $m_1$ with respect to $\psi$.  Finally, we will show that this monochromatic monotone 2-path will correspond to a topological subgraph that is weakly isomorphic to $C_{m_1}$. The detailed argument follows.

For integers $t\geq 0$ and $1 < a,b < m_2$, we construct a vertex subset $V^{a,b}_t \subset \{v_1,\dots, v_{n-1}\}$, an edge set $E^{a,b}_t$ of pairs in $V^{a,b}_t$, and a subset $S_t\subset \{v_1,\dots, v_{n-1}\}$ such that the following holds.
\begin{enumerate}

\item We have $\sum\limits_{1 < a,b < m_2} |V^{a,b}_t| = t$.

\item For all $1 < a,b< m_2$, we have $V^{a,b}_t \prec S_t$.

\item For $u_1 \in V^{a,b}_t$, we have $\phi(u_1,u_2) =(a,b)$ for every $u_2 \in V^{a,b}_t\cup S_t$ with $u_1 \prec u_2$.
    
\item  For each edge $\{u_1,u_2\} \in E^{a,b}_t$, where $u_1 \prec u_2$, we have $\chi(u_1,u_2,u_3) = \chi(u_1,u_2,u_{4})$ for all $u_3,u_{4} \in V^{a,b}_t$ such that $u_1 \prec u_2\prec u_3\prec u_{4}$.
 
\end{enumerate}

\noindent We start by setting $V^{a,b}_0=\emptyset$ for all $1 < a,b < m_2$, and $S_0 = \{v_1,\dots, v_{n-1}\}$. After stage $t$, we have $V^{a,b}_t$, $E^{a,b}_t$, for $1 < a,b < m_2$, and $S_t$ as described above.  
 
At the beginning of stage $t + 1$, let $w_{t+1}$ be the smallest element in $S_t$ with respect to $\prec$.  By the pigeonhole principle, there exists integers $1 < \alpha,\beta < m_2$ and a subset $S_{t,0}\subset S_t\setminus\{ w_{t+1}\}$ of size at least $(|S_t|-1)/m_2^2$, such that $\phi(w_{t+1}, u) = (\alpha,\beta)$ for all $u \in S_{t,0}$. Then we set $V^{\alpha,\beta}_{t+1}:= V^{\alpha,\beta}_t\cup \{w_{t+1}\}$. For all $1<a,b<m_2$ with $(a,b)\neq (\alpha,\beta)$, we set $V^{a,b}_{t+1}:=V^{a,b}_t$ and $E^{a,b}_{t+1}:=E^{a,b}_t$.

\begin{claimlabelled}\label{quasixmonotone}
For all $u\in V^{\alpha,\beta}_t$ and $v\in S_{t,0}$, we have $\chi(u,w_{t+1},v)\in \{000,010\}$.
\end{claimlabelled}
\begin{proof}
For the sake of contradiction, suppose $\chi(u,w_{t + 1},v) = 100$, where $u\in V^{\alpha,\beta}_t$ and $v\in S_{t,0}$.  Since $\phi(u,w_{t + 1}) = (\alpha,\beta)$, the longest monotone 3-path in color 100 ending at $\{u,w_{t + 1}\}$ has length $\alpha$.  Hence, the longest monotone 3-path in color 100 ending at $\{w_{t + 1},v\}$ has length at least $\alpha + 1$.  This contradicts the fact that $\phi(w_{t + 1},v) = (\alpha,\beta)$.  A similar argument follows if $\chi(u,w_{t + 1},v) = 001$.\qed\end{proof}

Now that we have constructed $V^{\alpha,\beta}_{t+1}$ by adding $w_{t+1}$ to $V^{\alpha,\beta}_t$, we play the vertex online Ramsey game so that builder chooses and creates edges of the form $\{u,w_{t + 1}\}$, where $u \in V^{\alpha,\beta}_t$, according to his strategy. After each edge $\{u,w_{t + 1}\}$ is created, painter immediately colors it $\psi(u,w_{t + 1}) \in \{000, 010\}$ as follows. In painter's strategy, after the $j$-th edge $\{u_j,w_{t+1}\}$ is created and colored, a set $S_{t,j}\subset S_{t,0}$ will be constructed such that all triples $\{u_j,w_{t+1},v\}$ with $v\in S_{t,j}$ are colored by $\chi$ with the same color in $\{000,010\}$. After the $(j+1)$-th edge $\{u_{j+1},w_{t+1}\}$ is created, painter looks at all triples of the form $\{u_{j+1},w_{t + 1},v\}$ with $v \in S_{t,j}$. Since $\chi(u_{j+1},w_{t + 1},v) \in \{000,010\}$ by Claim~\ref{quasixmonotone}, the pigeonhole principle implies that there exists a subset $S_{t,j+1}\subset S_{t,j}$ with size at least $|S_{t,j}|/2$ such that all triples $\{u_{j+1},w_{t+1},v\}$ with $v\in S_{t,j+1}$ are colored by $\chi$ with the same color $xyz\in \{000,010\}$. Then painter sets $\psi(u_{j+1},w_{t + 1}) = xyz$.

If builder decides to stop creating edges from $w_{t + 1}$ to $V^{\alpha,\beta}_t$ after $j$ edges are created and colored, the stage ends and we set $S_{t+1} = S_{t,j}$, and we let $E^{\alpha,\beta}_{t+1}$ be the union of $E^{\alpha,\beta}_t$ and all edges built during this stage. Let $e_{t+1}$ denote the total number of edges builder creates in stage $t + 1$. Recall that $e_{t+1} \geq 1$ unless $V^{\alpha,\beta}_t=\emptyset$. As long as $|S_{t+1}| > 0$, we continue this construction process by starting the next stage. Clearly, $V^{a,b}_{t+1}$, $E^{a,b}_{t+1}$, for all $1 < a,b < m_2$, and $S_{t+1}$ have the four properties described above.  We now make the following claim.
\begin{claimlabelled}
For $t\geq 1$, we have
\[|S_t| \geq \frac{n -1 }{ m_2^{2t}\cdot 2^{\sum_{i = 2}^{t} e_i}} - \sum\limits_{i = 2}^{t} \frac{1}{m_2^{2(t+1-i)}\cdot 2^{\sum_{j = i}^{t} e_j}}.\]
\end{claimlabelled}
\begin{proof}
We proceed by induction on $t$. For the base case $t = 1$, there's no edge for the builder to build in the first stage, so $|S_1| = |S_{0,0}|\geq (n-1)/m_2^2$ as desired. For the inductive step, assume the statement holds for $t\geq 1$. When we start stage $t + 1$ and introduce vertex $w_{t + 1}$, the set $S_t$ shrinks to $S_{t,0}$ whose size is guaranteed to be at least $(|S_t|-1)/m_2^2$, and each time builder creates an edge from $w_{t + 1}$ to $V^{\alpha,\beta}_t$, our set decreases by a factor of two.  Since builder creates $e_{t + 1}$ edges during stage $t + 1$, we have
\begin{align*}
    |S_{t + 1}|& \geq \frac{|S_t| - 1}{m_2^22^{e_{t+1}}} \geq \frac{n -1 }{ m_2^{2(t+1)}\cdot 2^{\sum_{i = 2}^{t+1} e_i}} - \sum\limits_{i = 2}^{t} \frac{1}{m_2^{2((t+1)+1-i)}\cdot 2^{\sum_{j = i}^{t+1} e_j}}-\frac{1}{m_2^22^{e_{t+1}}}\\
    &=  \frac{n -1 }{ m_2^{2(t+1)}\cdot 2^{\sum_{i = 2}^{t+1} e_i}} - \sum\limits_{i = 2}^{t+1} \frac{1}{m_2^{2((t+1)+1-i)}\cdot 2^{\sum_{j = i}^{t+1} e_j}},
\end{align*}which is what we want.\qed\end{proof}

After $t$ stages, builder has created a total of $\sum_{i = 1}^{t} e_i$ edges, such that each edge has color $000$ or $010$ with respect to $\psi$.  If there is no monochromatic 2-path of length $m_1$ with respect to $\psi$ on any $(V^{a,b}_t,E^{a,b}_t)$, this implies that

\[\sum_{i = 1}^{t} e_i < m_2^2 r(m_1) \leq 2(m_1m_2)^2\log m_1.\]

\noindent Also, since $e_i\geq 1$ for all but $m_2^2$ many indices $1\leq i\leq t$, we have \[t\leq m_2^2+\sum_{i = 1}^{t} e_i< 3(m_1m_2)^2\log m_1.\]

\noindent Since we assumed

\[n > 2^{9(m_1 m_2)^2\log(m_1)\log(m_2)},\]

\noindent we have\begin{align*}
    |S_t| &\geq \frac{n -1 }{ m_2^{2t}\cdot 2^{\sum_{i = 2}^{t} e_i}} - \sum\limits_{i = 2}^{t} \frac{1}{m_2^{2(t+1-i)}\cdot 2^{\sum_{j = i}^{t} e_j}}\\
    &\geq \frac{n -1 }{2^{8(m_1m_2)^2\log (m_1)\log (m_2)}} - \sum\limits_{i = 2}^{t} \frac{1}{2^{t- i  + 1}} > 1.
\end{align*}

Hence, we can continue to the next stage and introduce vertex $w_{t  + 1}$.  Therefore, when this process stops, say at stage $s$, we must have a monochromatic monotone 2-path of length $m_1$ with respect to $\psi$ on some $(V^{a,b}_s,E^{a,b}_s)$.   

Now let $W^* = \{w_1^*,\dots, w_{m_1}^*\}$, where $w_1^*\prec \cdots \prec w_{m_1}^*$, be the vertex set that induces a monochromatic monotone 2-path of length $m_1$ with respect to $\psi$ on $(V^{a,b}_s,E^{a,b}_s)$. Since $\phi$ colors every pair in $W^*$ with the color $(a,b)$, by following the proof of Claim~\ref{quasixmonotone}, we have $\chi(w_i^*,w_j^*,w_k^*)  \in \{000,010\}$ for every $i < j < k$. Hence, the following argument due to Pach, Solymosi, and T\'oth~\cite{pst} shows that $W^*$ induces a topological subgraph that is weakly isomorphic to $C_{m_1}$. For the sake of completeness, we include the proof.
\begin{claimlabelled}\label{PSTlemma}
Let $W^{*} = \{w^*_1,\dots, w^*_{m_1}\}$ be as described above. Then $W^*$ induces a topological subgraph that is weakly isomorphic to $C_{m_1}$.
\end{claimlabelled}
\begin{proof}
Suppose $\psi(w^*_i,w^*_{i + 1}) = 000$ for all $i$. It suffices to show that every triple in $W^*$ has color 000 with respect to $\chi$.  For the sake of contradiction, suppose we have $w^*_i \prec w^*_j \prec w_k^*$ such that $\chi(w^*_i,w^*_j, w^*_k) = 010$, and let us assume that $j-i$ is minimized among all such examples. Since $\{w_{i}^*,w_{i+1}^*\}\in E^{a,b}_s$, we have $\chi(w^*_{i},w^*_{i + 1},w^*_k) =\psi(w^*_i,w^*_{i + 1}) = 000$. This implies that $j>i+1$ and the edge $w^*_{i + 1}w^*_k$ crosses $v_0w^*_j$ (see Figure~\ref{fig:PSTlemma}), which contradicts the minimality condition. A similar argument follows if $\psi(w^*_i,w^*_{i + 1}) = 010$ for all $i$.\qed\end{proof}
This completes the proof of Theorem \ref{main2}\qed\end{proof}
\begin{figure}
    \centering
    \includegraphics[scale=.5]{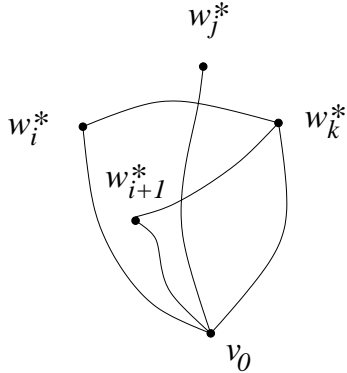}
    \caption{A figure illustrating Claim~\ref{PSTlemma}.}
    \label{fig:PSTlemma}
\end{figure}

\section{Plane Path}

In this section, we prove Theorem~\ref{path}. We will need the following lemma, which was observed by Fulek and Ruiz-Vargas in \cite{fr}.
\begin{lemma}\label{K2n}
If a complete simple topological graph $G$ contains a topological subgraph that is isomorphic to a plane $K_{2,m^2}$, then $G$ contains a plane path of length $\Omega(m)$.
\end{lemma}
\noindent Let us briefly explain how to establish this lemma, as it is not explicitly stated in \cite{fr}.  In \cite{toth}, T\'oth proved that every $n$-vertex geometric graph with more than $2^9k^2n$ edges contains $k$ pairwise disjoint edges. His proof easily generalizes to simple topological graphs whose edges are drawn as $x$-monotone curves, and, in fact, shows the existence of a plane path of length $2k$.  

Given a plane topological subgraph $K_{2,m^2}$ inside a complete simple topological graph $G$, Fulek and Ruiz-Vargas \cite{fr} showed that there exists a topological subgraph $G'\subset G$, with $m^2$ vertices and $\Omega(m^4)$ edges, that is weakly-isomorphic to an $x$-monotone simple topological graph $G''$. Hence, we can conclude Lemma~\ref{K2n} by applying T\'oth's result stated above with $k = \Omega(m)$.

\begin{proof}[of Theorem~\ref{path}]
First, we keep the following notations from the proof of Theorem~\ref{main}. Let $G = (V, E)$ be a complete $n$-vertex simple topological graph. We can assume that there is a vertex $v_0 \in  V$ such that $v_0$ lies on the boundary of the unbounded cell. We label the other vertices by $v_1,\dots,v_{n-1}$ such that the edges $v_0v_i$, for $1\leq i<n$, emanate out from $v_0$ in clockwise order. We write $v_i \prec v_j$ if $i < j$, and color every triple $v_i\prec v_j \prec v_k$ by $\chi(v_i,v_j,v_k)\in \{000,010,100,001\}$.

For each $v_i$, we arrange the vertices $\{v_{i+1},\dots,v_{n-1}\}$ into a sequence $\theta(v_i)=(v_{j_1},\dots, v_{j_{n-1-i}})$ such that the topological edges $v_iv_0, v_iv_{j_1}, v_iv_{j_2}, \dots, v_{i}v_{j_{n-1-i}}$ emanate out from $v_i$ in counterclockwise order. See Figure~\ref{fig:radial_order}. We call a sequence of vertices $S=(v_{i_1},\dots,v_{i_k})$ increasing (or decreasing) if $v_{i_1}\prec v_{i_2}\prec \dots\prec v_{i_k}$ (or $v_{i_1}\succ v_{i_2}\succ \dots\succ v_{i_k}$).

\begin{figure}
    \centering
    \scalebox{0.5}{\includegraphics{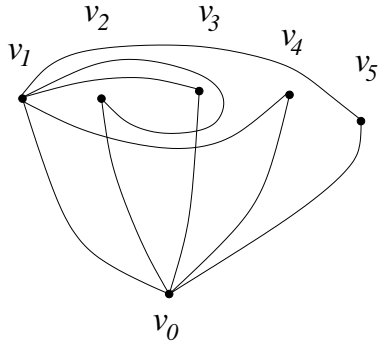}}
    \caption{An example with $\theta(v_1)=(v_4,v_3,v_2,v_5)$.}
    \label{fig:radial_order}
\end{figure}

\begin{lemma}\label{increasing}
If there exists a vertex $u$ such that $\theta(u)$ contains an increasing subsequence $(u_1,\dots,u_{m^2})$, then the edges $v_0u_{i}$ and $uu_{i}$, for all $1\leq i\leq m^2$, form a plane subgraph $K_{2,m^2}$.
\end{lemma}
\begin{proof}
It suffices to show $v_0u_i$ and $uu_j$ do not cross each other for every $1\leq i,j\leq k$. When $i=j$, this follows from $G$ being simple. When $j>i$, by the increasing assumption, the edges $uv_0$, $uu_i$, and $uu_j$ emanate out from $u$ in counterclockwise order. Observe that this condition forces $u_j$ to be outside the region $\Delta_{v_0uu_i}$ bounded by the topological edges $v_0u$, $uu_{i}$, and $u_{i}v_0$. Then the Jordan arc $uu_j$ starting at $u$, initially outside $\Delta_{v_0uu_i}$, cannot enter $\Delta_{v_0uu_i}$ then leave again to end at $u_j$. In particular, $uu_j$ doesn't cross $v_0u_i$. See Figure~\ref{fig:increasing} for an illustration. A similar argument follows if $j<i$.
\qed\end{proof}

\begin{figure}[ht]
    \centering
    \scalebox{0.5}{\includegraphics{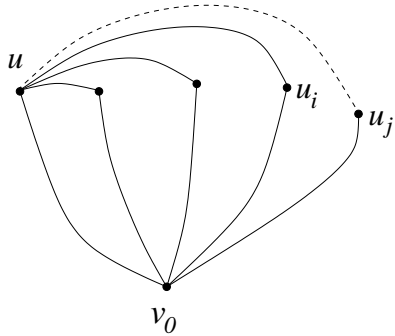}}
    \caption{An increasing subsequence of $\theta(u)$ induce a plane $K_{2,m^2}$.}
    \label{fig:increasing}
\end{figure}

We set $m = \left\lfloor\frac{\log n}{2\log\log n}\right\rfloor$ and prove that $G$ contains a plane path of length $\Omega(m)$. We can assume $m>1$, otherwise there's nothing to prove. If some sequence $\theta(v_i)$ contains an increasing subsequence of length $m^2$, then by Lemma~\ref{increasing} and Lemma~\ref{K2n}, we are done. Therefore, we assume that $\theta(v_i)$ doesn't contain an increasing subsequence of length $m^2$ for every $i$.

For integer $t\geq 1$, we inductively construct subsets $U_t, S_t\subset \{v_1,\dots, v_{n-1}\}$ with $U_t=\{u_1,\dots,u_t\}$, where $u_1\prec\dots\prec u_t$, and $U_t\prec S_t$. Initially we set $U_1=\{u_1:=v_1\}$ and $S_1=\{v_2,\dots, v_{n-1}\}$. Suppose for some $t$, we have already constructed $U_t$ and $S_t$. If $|S_t|\leq m^2$, we stop this construction process, otherwise we continue to construct $U_{t+1}$ and $S_{t+1}$ as follows: Let $\theta'$ be the subsequence of $\theta(u_{t})$ that contains exactly those vertices in $S_{t}$. Note that the length of $\theta'$ equals to $|S_t|$. According to our assumption, the length of the longest increasing subsequence in $\theta'$ is less than $m^2$. Hence, by Dilworth's theorem \cite{dilworth}, $\theta'$ contains a decreasing subsequence of length at least $|S_t|/m^2$. Let $S'_{t+1}$ be the set of vertices that appear in this decreasing subsequence of $\theta'$. Next, we take $u_{t+1}$ to be the smallest element of $S'_{t+1}$ with respect to $\prec$ and let $U_{t+1}:=U_t\cup\{u_{t+1}\}$. Consider the region $\Delta_{v_0u_{t}u_{t+1}}$ bounded by the topological edges $v_0u_{t}$, $u_{t}u_{t+1}$, and $u_{t+1}v_0$. Each vertex in $S'_{t+1}\setminus \{u_{t+1}\}$ is either inside or outside $\Delta_{v_0u_{t}u_{t+1}}$. So, by the pigeonhole principle, there exists a subset $S_{t+1}\subset S'_{t+1}\setminus \{u_{t+1}\}$ with $|S_{t+1}|\geq |S'_{t+1}\setminus \{u_{t+1}\}|/2$ such that the whole set $S_{t+1}$ is either inside or outside $\Delta_{v_0u_{t}u_{t+1}}$. Clearly, we have $U_{t+1}\prec S_{t+1}$ and
\[|S_{t+1}|\geq \frac{|S_{t}|/m^2-1}{2}\geq \frac{|S_{t}|}{(2m)^2}.\]

Using the inequality above and the fact that $|S_1|=n-2$, we can inductively prove $|S_t|\geq \frac{n}{(2m)^{2t}}$. When $t= m-1$, this gives us
\[|S_{m-1}|\geq \frac{n}{(2m)^{2(m-1)}} >  \frac{n}{(2m)^{\log n/\log\log n-2}}> m^2\cdot\frac{n}{(\log n)^{\log n/\log\log n}}\geq m^2.\]
Hence, the construction process ends at a certain $t$ larger than $m-1$, and we will always construct $U_{m}=\{u_1,\dots,u_{m}\}$.

Now we show that $u_iu_{i+1}$, for $1\leq i< m$, form a plane path. Our argument is based on the following two claims.
\begin{claimlabelled}\label{tri-region}
For any vertices $u_i\prec u_{i+1}\prec u_j\prec u_k$, we have $u_j$ and $u_k$ either both inside or both outside the region $\Delta_{v_0u_{i}u_{i+1}}$.
\end{claimlabelled}
Claim~\ref{tri-region} is obviously guaranteed by the construction process of $U_m$.
\begin{claimlabelled}\label{decreasing}
For any vertices $u_i\prec u_j\prec u_k$, the topological edges $v_0u_i$ and $u_ju_k$ do not cross each other.
\end{claimlabelled}
\begin{proof}
Consider the region $\Delta_{v_0u_iu_j}$ bounded by the topological edges $v_0u_i$, $u_iu_j$, and $u_jv_0$, then $u_k$ is either inside or outside $\Delta_{v_0u_iu_j}$. If $u_k$ is inside $\Delta_{v_0u_iu_j}$, then $v_0u_k$ must cross $u_iu_j$. By Observation~\ref{observation_pst}, we have $\chi(u_i,u_j,u_k)=001$, which implies $v_0u_i$ and $u_ju_k$ do not cross. See the third configuration in Figure~\ref{fig:colors}.

Suppose $u_k$ is outside $\Delta_{v_0u_iu_j}$. By the construction process of $U_m$, the edges $u_iv_0$, $u_iu_k$ and $u_iu_j$ must emanate from $u_i$ in counterclockwise order, this implies that $u_iu_k$ crosses $v_0u_j$. Then, by Observation~\ref{observation_pst}, $\chi(u_i,u_j,u_k)=010$ and $u_ju_k$ doesn't cross $v_0u_i$. See the second configuration in Figure~\ref{fig:colors}.
\qed\end{proof}

\begin{figure}[ht]
    \centering
    \scalebox{0.4}{\includegraphics{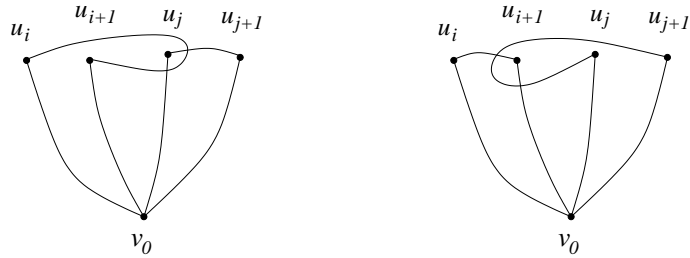}}
    \caption{For $u_iu_{i+1}$ and $u_ju_{j+1}$ with $i+1<j$ to cross each other, either $u_j$ and $u_{j+1}$ are not both inside or both outside $\Delta_{v_0u_iu_{i+1}}$ (left graph), or the topological edge $u_ju_{j+1}$ crosses one edge in $\{v_0u_i,v_0u_{i+1}\}$ (right graph).}
    \label{fig:planepath}
\end{figure}

Finally, we argue that the edges $u_iu_{i+1}$ and $u_ju_{j+1}$ do not cross for any $i<j$. When $j=i+1$, this follows from $G$ being simple. When $j>i+1$, by Claim~\ref{tri-region}, the vertices $u_j$ and $u_{j+1}$ are either both inside or both outside the region $\Delta_{v_0u_iu_{i+1}}$. So, the edge $u_ju_{j+1}$ crosses the boundary of $\Delta_{v_0u_iu_{i+1}}$ an even number of times. On the other hand, by Claim~\ref{decreasing}, $u_ju_{j+1}$ doesn't cross $v_0u_i$ or $v_0u_{i+1}$. So $u_ju_{j+1}$ doesn't cross $u_iu_{i+1}$. See Figure~\ref{fig:planepath} for an illustration. This concludes the proof of Theorem~\ref{path}.
\qed\end{proof}

\section{Concluding Remarks}

Answering a question of Pach and T\'oth \cite{pt05}, Suk showed that every complete $n$-vertex simple topological graph contains $\Omega(n^{1/3})$ pairwise disjoint edges \cite{suk} (see also \cite{fr}).  This bound was later improved to $n^{1/2 - o(1)}$ by Ruiz-Vargas in \cite{ruiz}. Hence, for plane paths, we conjecture a similar bound should hold.

\begin{conjecture}
There is an absolute constant $\varepsilon > 0$, such that every complete $n$-vertex simple topological graph contains a plane path of length $n^{\varepsilon}$.
\end{conjecture}

Let $h = h(n)$ be the smallest integer such that every complete $n$-vertex simple topological graph contains an edge crossing at most $h$ other edges.  A construction due to Valtr (see page 398 in \cite{bmp}) shows that $h(n) \geq \Omega(n^{3/2})$.  In the other direction, Kyn\v{c}l and Valtr \cite{kv} used an asymmetric version of Theorem~\ref{old} to show that $h(n) = O(n^2/\log^{1/4}n)$. By using Theorem \ref{main2} instead, their arguments show that $h(n) \leq n^2/(\log n)^{1/2 - o(1)}$.  We conjecture the following.

\begin{conjecture}
There is an absolute constant $\varepsilon > 0$ such that $h(n) \leq n^{2 -\varepsilon}$.
\end{conjecture}

\end{document}